\documentclass{amsart}
\usepackage{amsmath}
\usepackage{graphicx}
\usepackage{amsfonts}
\usepackage{amsthm}
\usepackage{centernot}
\usepackage{setspace}
\usepackage{mathtools}
\usepackage{amssymb}
\usepackage{bm}
\usepackage{hyperref}
\usepackage[capitalize]{cleveref}
\usepackage{mathrsfs}
\usepackage{tikz}
\usepackage{tikz-cd}
\usepackage{breqn}
\usepackage[strict]{changepage}
\usepackage{breqn}
\newtheorem{thm}{Theorem}[section]
\newtheorem{cor}[thm]{Corollary}
\newtheorem{prop}[thm]{Proposition}
\newtheorem{lem}[thm]{Lemma}

\theoremstyle{definition}

\theoremstyle{remark}

\makeatletter
\let\c@equation\c@thm
\makeatother
\numberwithin{equation}{section}

\newcommand{\ov}[1]{\,\overline{#1}}

\DeclareMathOperator{\ad}{ad}

\DeclareMathOperator{\GL}{GL}
\DeclareMathOperator{\PGL}{PGL}
\DeclareMathOperator{\PSL}{PSL}

\DeclareMathOperator{\PSU}{PSU}

\DeclareMathOperator{\Ind}{Ind}

\DeclareMathOperator{\tr}{tr}

\DeclareMathOperator{\Gal}{Gal}

\DeclareMathOperator{\Frob}{Frob}

\newcommand{\eps}{\varepsilon}

\newcommand{\Z}{\mathbf{Z}}
\newcommand{\Q}{\mathbf{Q}}

\newcommand{\C}{\mathbf{C}}

\newcommand{\Qbar}{\ov{\Q}}
\newcommand{\F}{\mathbf{F}}

\title[Potential automorphy in dimension 3]{Potential automorphy of certain non self-dual 3-dimensional Galois representations}
\author{Konstantin Miagkov} 
\date{}
\begin{document}

\begin{abstract} 
In a series of papers, van Geemen and Top have defined a family of surfaces $S_z$ indexed by a nonzero integer parameter $z$, and a compatible family of 3-dimensional Galois representations over $\Q(i)$ attached to each surface. In this note we use recent advancements in potential automorphy and automorphy lifting to show that these compatible families are potentially automophic for all values of $z$, and hence that their L-functions have analytic continuation and a functional equation.
\end{abstract}

\maketitle

\tableofcontents

\section{Introduction}
\subsection{Results}

Fix an isomorphism $\iota_{l} \colon \overline{\Q}_{l} \cong \C$ for every rational prime $l$.
Also choose a positive integer $z \geq 2$. By decomposing the transcendental part of the $l$-adic cohomology of
a projective smooth surface $S$ whose affine model is
\[ t^2 = xy(x^2 - 1)(y^2 - 1)(x^2 - y^2 + zxy), \]
van Geemen-Top constructed three-dimensional $l$-adic representations
\[ \rho_{\lambda} \colon G_{\Q} \to \GL_3(\Q(i)_{\lambda}). \]
Here $\lambda$ is a finite place of $\Q(i)$. This provides a concrete example of a non self-dual compatible system.
Our main theorem is the following:
\begin{thm}\label{thm:main}
    For any nonzero integer $z$, the compatible system $\{\rho_{\lambda}\}$ is potentially automorphic.
\end{thm} 
We also record the following corollary:
\begin{cor} The Hasse-Weil L-function of the surface $S_z$ has analytic continuation to the complex plane, and it satisfies a functional equation.
\end{cor}
\subsection{van Geemen-Top family}

Here we give a brief overview of the construction of of $\rho_\lambda$; for details, see \cite{vGT94}.
Choose a positive integer $z \geq 2$, and let $S$ be a surface defined over $\Q$ by the equation 
\[ t^2 = xy(x^2 - 1)(y^2 - 1)(x^2 - y^2 + zxy). \]
Consider the Galois $l$-adic representations on \'etale cohomology:
\[\sigma_l: G_{\Q} \to \GL(H_{\text{\'et}}(S_{\Qbar},\Q_l)).\]
The Galois representation decomposes:
\[H_{\text{\'et}}(S_{\Qbar},\Q_l) = T_l \oplus \text{NS}(S_{\Qbar}) \otimes_\Z \Q_l,\]
where $\text{NS}(S_{\Qbar})$ is the N\'eron-Severi group and $T_l$ its orthogonal complement with
respect to the cup product: $H^2 \times H^2 \to H^4 \cong \Q_l$ . Eigenvalues of Frobenius
on $\text{NS}(S_{\Qbar})$ are roots of unity multiplied by $p$, so we will look for our family of $\lambda$-adic representations inside $T_l \otimes_{\Q_l} \Q(i)_l$. The surface $S$ has an automorphism $\varphi$ defined over $\Q$ given by $(x,y,t) \mapsto (y,-x,t)$. This automorphism satisfies $\varphi^4 = 1$. The induced map: $\varphi_*: H^2 \to H^2$ commutes with the Galois representations. We have $\dim T_l = 6$ and $\varphi_*: T_l \to T_l$ has two three dimensional eigenspaces $W_{\lambda}, W_{\lambda'}$ with eigenvalues $i, -i$ respectively, such that:
\[T_\lambda \coloneqq T_l \otimes_\Q \Q_l(i) = W_\lambda \oplus W_{\lambda'}.\]
This decomposition gives us a compatible family of $\lambda$-adic three dimensional Galois representations $\rho_{\lambda}'$ on $W_\lambda$, a $Q(i)_\lambda$ vector space. In general $\rho_\lambda'$ has to be twisted by a Dirichlet character $\eps$ to obtain $\eps \otimes \rho_\lambda' = \rho_\lambda$ with $\det \rho_\lambda = \chi^3$, where $\chi$ is the cyclotomic character. 
These are the representations $\rho_\lambda : G_\Q \to \GL(W_{\lambda})$ that we will consider. These representations form a weakly compatible system in the sense of \cite[\S 5.1]{BLGGT} which is pure of weight 1 and has Hodge-Tate weights $\{0,1,2\}$. We denote this compatible system by $\mathcal{R}$.

\subsection{Acknowledgments}

The result of this note may very well be known to many experts in the field. One of the main ideas of the proof appeared in Frank Calegari's blog post \cite{CalegariBlog}. We are also greatful to Richard Taylor, Vaughan McDonald and Stepan Kazanin for helpful conversations.

%\section{Local representations}\label{sec:local}
\section{Proof of \cref{thm:main}}
The characteristic polynomial of $\rho_{\lambda}(\Frob_p)$ does not depend on $\lambda$, and we may write it as $Q_p(X) = X^3 - a_pX^2 + \ov{a_p}pX - p^3$. The roots of this polynomial as Weil numbers with absolute value $p$, so we have a bound $\lvert a_p \rvert \leq 3p$.
We recall the following result \cite[Proposition~5.3.2]{BLGGT}:
\begin{prop}\label{big} Suppose that $\mathcal{R}$ is a regular, weakly
    compatible system of $l$-adic representations of $G_F$ defined over
    $M$. 
  If $s$ is a subrepresentation of $r_\lambda$ then we will write $\ov{s}$ for the semi-simplification of the reduction of $s$. Also write $l$ for the rational prime below $\lambda$. Then there is a set of rational primes $\mathcal{L}$ of Dirichlet density $1$ (depending only on $\mathcal{R}$), such that if $s$ is any irreducible subrepresentation of $r_\lambda$ for any $\lambda$ dividing any element of $\mathcal{L}$ then $\ov{s}|_{G_{F(\zeta_l)}}$ is irreducible.
\end{prop}
From this point on, we restrict our attention to the places in $\mathcal{L}$.
\subsection{Irreducible case}
In this section, we assume that $\rho_\lambda$ is absolutely irreducible at all places $\lambda \in \mathcal{L}$. Our main tool is the following potential automorphy theorem:
\begin{thm}
\label{thm:pot_aut}
    Suppose $F$  is a CM number field, $F^{\emph{av}}$  is a finite extension of $F$ and $n\geq 2$  is a positive integer. Let $p$  be a prime number. Suppose that
    \[
    r: G_F\rightarrow \GL_n(\overline{\Q}_l)
    \]
    is a continuous representation satisfying the following condition:
    \begin{enumerate}
        \item $r$ is unramified almost everywhere.
        \item For each place $v\mid l$ of $F$,  the representation $r|_{G_{F_{v}}}$  is potentially semistable, ordinary with regular Hodge-Tate weights.
        \item $\ov{r}$  is absolutely irreducible and decomposed generic (cf. \cite[Definition~4.3.1]{10author}). The image of $\ov{r}|_{G_{F(\zeta_{l})}}$  is adequate (cf. \cite[Definition~1.1]{MT23}).
        \item There exists $\sigma\in G_{F}-G_{F(\zeta_{l})}$  such that  $\ov{r}(\sigma)$ is a scalar.
    \end{enumerate}
        
Then there exists a finite CM Galois extension $F'/F$  linearly disjoint from $F^{\emph{av}}$  over $F$ such that $r|_{G_{F'}}$  is ordinarily automorphic.
\end{thm}
\begin{proof}
    This is almost \cite[Theorem~1.4]{Lie}, except the ``enormous'' condition on the image is replaced with ``adequate''. This is achieved by using \cite[Theorem~1.3]{MT23} instead of \cite[Theorem~6.1.2]{10author} as the autormophy lifting input to the arguments of \cite{Lie}.
\end{proof}

\begin{lem}\label{lem:residual_good}
    Let $\lambda \in \mathcal{L}$ lie above a rational prime $l$. Then for all but finitely many $\lambda$
    \begin{enumerate}
        \item $\ov{\rho}_\lambda|_{G_{\Q(\zeta_{l})}}$ is adequate.
        \item There exists $\sigma\in G_{\Q}-G_{\Q(\zeta_{l})}$  such that  $\ov{\rho}_\lambda(\sigma)$ is a scalar.
    \end{enumerate}
\end{lem}
\begin{proof}
    (1) follows from \cite[Theorem~9]{adequate}. For (2), note that $\Gal(\Q(\zeta_l)/\Q)\cong (\Z/l\Z)^\times$, and so it would be enough to show that $(\ad \ov{\rho}_\lambda)(G_{\Q})$ does not surject onto $(\Z/l\Z)^\times$. From \iffalse \cite[Theorem~47]{pgl3}  \cite[\S 15]{Mitchell} \fi \cite[\S 6]{DN} we have the following possibities for $G_\lambda \coloneqq (\ad \ov{\rho}_\lambda)(G_{\Q}) \subset \PGL_3(\ov{\F}_l)$ for all but finitely many $\lambda$:
    \begin{enumerate}
        \item $\PSL_3(\F_l)$ if $l \equiv 1 \pmod 4$.
        \item $\PSU_3(\F_l)$ if $l \equiv 3 \pmod 4$
        \item $G_\lambda$ is contained in the normalizer of a cyclic group, and fits into an exact sequence \[1 \to \mathcal{C} \to G_\lambda \to C_3 \to 1,\] where $\mathcal{C}$ is cyclic of order dividing $l^2+l+1$.
        \item $G_\lambda$ is contained in the normalizer of a diagonal subgroup, and fits into an exact sequence \[1 \to \mathcal{R} \to G_\lambda \to \mathcal{U} \to 1,\] where $\mathcal{R}$ is diagonal of order dividing $(l-1)^2$, and $\mathcal{U}$ is either $S_3$ or a subgroup generated by a 3-cycle in $S_3$. Here the action of $\mathcal{U}$ on $\mathcal{R}$ is the restriction of the usual action of the Weyl group on the maximal torus. 
    \end{enumerate}
    None of these groups surject to $(\Z/l\Z)^\times$ for $l > 7$, hence we are done.
\end{proof}

For any positive number $r$ let $C_r \subset \Z[i]$ denote the subset of Gaussian integers with absolute value $\leq r$. For any complex number $x$ let $xC_r \coloneqq \{xy \mid y \in C_r\}$.
\begin{prop}\label{prop:irreducible}
    In the irreducible case the compatible system $\mathcal{R}$ is potentially automorphic.
\end{prop}
\begin{proof}
    We begin by showing that there exists a positive density set of places $\mathcal{P}$ of $\Q(i)$ such that $\rho_v$ is ordinary for all $v \in \mathcal{P}$. Let $p$ be a prime split in $\Q(i)$, and let $v_1, v_2$ be the places of $\Q(i)$ above $p$. Suppose that $\rho_{v_1}$ and $\rho_{v_2}$ are both not ordinary. Then $p \mid a_p$, which implies that $a_p \in pC_3$. Similarly $a_p \in pC_3$ if $p$ is inert in $\Q(i)$ and $\rho_p$ is not ordinary. Thus if $\rho$ is not ordinary at at any place above $p$, then for any $\lambda \mid l$ we have $\tr (\chi^{-1} \otimes \rho_\lambda)(\Frob_p) \in C_3$. From \cite[\S 5]{DN} we can list the following possibilities for $({\ov{\rho}_\lambda} \otimes \ov{\chi}^{-1})|_{I_l}$:

    \[\begin{pmatrix} \psi_1^{-1} & * & * \\ 0 & 1 & * \\ 0 & 0 & \psi_1 \end{pmatrix}, \quad 
    \begin{pmatrix}\psi_1^{-1} & * & * \\ 0 & \psi_2 & 0 \\ 0 & 0 & \psi_2^l \end{pmatrix}, \quad
    \begin{pmatrix}1 & * & * \\ 0 & \psi_2^{1-l} & 0 \\ 0 & 0 & \psi_2^{l-1} \end{pmatrix},\]

    \[\begin{pmatrix} \psi_1 & * & * \\ 0 & \psi_2^{-l} & 0 \\ 0 & 0 & \psi_2^{-1} \end{pmatrix}, \quad 
    \begin{pmatrix}\psi_3^{l-l^2} & 0 & 0 \\ 0 & \psi_3^{l^2-1} & 0 \\ 0 & 0 & \psi_3^{1-l} \end{pmatrix}.\]
    Here $\psi_n$ is the fundamental character of level $n$. It is easy to see that each of these images contains $O(l)$ different traces. Thus for large enough $l$ there exists an element $\sigma \in G_{\Q}$ such that $\tr (\chi^{-1} \otimes \rho_\lambda)(\sigma) \notin C_3$. By the Chebotarev density theorem we thus have a positive density set of rational primes $\mathcal{P}$ such that $\rho_{v}$ is ordinary for at least one $v \mid p$. Combined with \cref{lem:residual_good} we see that we can find a place $v$ such that $\rho_v$ satisfies all conditions of \cref{thm:pot_aut}, and we are done.
\end{proof}

\subsection{Reducible case}
Now we assume that $\rho_\lambda$ is reducible for some place $\lambda$. If for any place $\lambda$ we have $\rho_\lambda \cong \chi_1 \oplus \chi_2 \oplus \chi_3$, where $\chi_i : G_\Q \to \ov{\Q}_l^\times$ are characters, then $\rho_\lambda$ is automorphic by class field theory. Thus we can assume $\rho_\lambda = \chi_\lambda \oplus s_\lambda$, where $\dim \chi_\lambda = 1$ and $\dim s_\lambda = 2$. The character $\chi_\lambda$ fits into a compatible system $\mathcal{X}$, and thus we have a splitting $\mathcal{R} = \mathcal{X} \oplus \mathcal{S}$ for a 2-dimensional compatible system $\mathcal{S}$. Let $b_p \coloneqq \tr s_\lambda(\Frob_p)$. Since $\rho_\lambda$ is pure of weight 1 and has Hodge-Tate weights $\{0, 1, 2\}$, it follows that $\chi_\lambda = \omega_l^{-1} \otimes \chi_f$, where $\omega_l$ is the $l$-adic cyclotomic character, and $\chi_f$ is a finite character. Then $s_\lambda$ has Hodge-Tate weights $
\{0, 2\}$. We now recall the following theorem of Calegari, \cite[Theorem~1.1]{CalegariFM}:
\begin{thm}\label{thm:fm}
    $\rho: G_{\Q} \rightarrow \GL_2(\ov{\Q}_l)$ be a continuous
    Galois representation which is unramified except at a finite number of primes. Suppose
    that $l > 7$, and, furthermore, that
    \begin{enumerate}
    \item $\rho|_{G_{\Q_l}}$ is potentially semi-stable, with distinct Hodge--Tate weights.
    \item The residual representation $\ov{\rho}$ is absolutely irreducible and not of dihedral type.
    \item $\ov{\rho}|_{G_{\Q_l}}$ is not a twist of a representation of the form
    $\displaystyle{ \left( \begin{matrix}\omega & * \\ 0 & 1  \end{matrix} \right)}$ where
    $\omega$ is the mod-$l$ cyclotomic character.
    \end{enumerate}
    Then $\rho$ is modular.
\end{thm}
Since $\chi_\lambda$ is automatically modular, it is enough to make sure $s_\lambda$ satisfies the hypotheses of \cref{thm:fm}. Since $s_\lambda$ has distinct non-consecutive Hodge-Tate weights, it satisfies (1) and (3). Thus we only need to show that we can choose $\lambda$ such that $\ov{s}_\lambda$ is not of dihedral type. 
\begin{lem}\label{lem:dihedral}
    Suppose we are in the reducible case. Then either for some $\lambda$ there exists a quadratic extension $E/\Q$ and a character $\psi_\lambda : G_E \to \ov{\Q}_l^\times$ such that $s_\lambda = \Ind_E^\Q \psi_\lambda$, or $\ov{s}_\lambda$ is not of dihedral type for all but finitely many $\lambda$.
\end{lem} 
\begin{proof}
    Let $N$ denote the product of all primes $p$ such that $\rho_\lambda$ is ramified at $p$ for at least one $\lambda$. Let $Q$ denote the set of all quadratic extensions of $G_\Q$ which are unramified a primes not dividing $N$. Each such extension $E/\Q$ has an associated quadratic character $\eps_{E}$. We will distinguish between the following two cases:\\
    Case 1: There exists an extension $E \in Q$ such that $b_p = 0$ for all $p \nmid lN$ inert in $E$. Then for all but finitely many primes $\tr s_\lambda(\Frob_p) = \tr (s_\lambda \otimes \eps_E)(\Frob_p)$, which implies $s_\lambda \cong s_\lambda \otimes \eps_E$. This in turn means that $s_\lambda \cong \Ind_E^\Q \psi_\lambda$ for some character $\psi_\lambda$ of $G_E$.\\
    Case 2: For each $E \in Q$ we can choose a prime $p_E$ inert in $E$ such that $p_E \nmid N$ and $b_{p_E} \neq 0$. Suppose $\ov{s}_\lambda$ is of dihedral type for some $\lambda$. Then there exists $E \in Q$ such that $\ov{s}_\lambda \cong \ov{s}_\lambda \otimes \eps_E$. This implies that $\lambda \mid b_{p_E}$. Therefore $\ov{s}_\lambda$ is not dihedral for $\lambda$ not dividing $\prod_{E \in Q} b_{p_E}$. 
\end{proof}
Putting together \cref{prop:irreducible}, \cref{thm:fm} and \cref{lem:dihedral} we are done with the proof of \cref{thm:main}.

\bibliographystyle{amsalpha}
\bibliography{refs}

\newcommand{\etalchar}[1]{$^{#1}$}
\providecommand{\bysame}{\leavevmode\hbox to3em{\hrulefill}\thinspace}
\providecommand{\MR}{\relax\ifhmode\unskip\space\fi MR }
\providecommand{\MRhref}[2]{%
  \href{http://www.ams.org/mathscinet-getitem?mr=#1}{#2}
}
\providecommand{\href}[2]{#2}
\begin{thebibliography}{BLGGT14}

\bibitem[ACC{\etalchar{+}}18]{10author}
Patrick~B. Allen, Frank Calegari, Ana Caraiani, Toby Gee, David Helm, Bao V.~Le
  Hung, James Newton, Peter Scholze, Richard Taylor, and Jack~A. Thorne,
  \emph{{Potential automorphy over CM fields}}, 2018.

\bibitem[BLGGT14]{BLGGT}
Thomas Barnet-Lamb, Toby Gee, David Geraghty, and Richard Taylor,
  \emph{Potential automorphy and change of weight}, Ann. of Math. (2)
  \textbf{179} (2014), no.~2, 501--609.

\bibitem[Cal12]{CalegariFM}
Frank Calegari, \emph{Even {G}alois representations and the {F}ontaine--{M}azur
  conjecture. {II}}, J. Amer. Math. Soc. \textbf{25} (2012), no.~2, 533--554.

\bibitem[Cal21]{CalegariBlog}
Frank Calegari, \emph{Potential automorphy for gl(n)}, 2021.

\bibitem[DV04]{DN}
Luis Dieulefait and Núria Vila, \emph{On the images of modular and geometric
  three-dimensional galois representations}, American Journal of Mathematics
  \textbf{126} (2004), no.~2, 335--361.

\bibitem[GHTT12]{adequate}
Robert Guralnick, Florian Herzig, Richard Taylor, and Jack Thorne,
  \emph{Adequate subgroups}, Appendix to [Tho12] (2012).

\bibitem[MT23]{MT23}
Konstantin Miagkov and Jack~A. Thorne, \emph{Automorphy lifting with adequate
  image}, Forum Math. Sigma \textbf{11} (2023), Paper No. e8, 31.

\bibitem[Qia23]{Lie}
Lie Qian, \emph{Potential automorphy for {$GL_n$}}, Invent. Math. \textbf{231}
  (2023), no.~3, 1239--1275.

\bibitem[Tho12]{Thorne12}
Jack Thorne, \emph{On the automorphy of {$l$}-adic {G}alois representations
  with small residual image}, J. Inst. Math. Jussieu \textbf{11} (2012), no.~4,
  855--920, With an appendix by Robert Guralnick, Florian Herzig, Richard
  Taylor and Thorne.

\bibitem[vGT94]{vGT94}
Bert van Geemen and Jaap Top, \emph{A non-selfdual automorphic representation
  of {${\rm GL}_3$} and a {G}alois representation}, Invent. Math. \textbf{117}
  (1994), no.~3, 391--401.

\end{thebibliography}
\nocite{Thorne12}

\end{document}